\documentclass[12pt]{iopart}

\expandafter\let\csname equation*\endcsname\relax
\expandafter\let\csname endequation*\endcsname\relax
\usepackage{iopams}  
\usepackage{ulem}
\usepackage{multirow}
\usepackage{amstext}
\usepackage{chemformula}
\usepackage{import}
\usepackage{amssymb}
\usepackage{amsthm}
\usepackage{amsmath, mathtools}
\usepackage{color}
\usepackage{enumerate,fancyhdr,verbatim,comment}
\usepackage{dsfont}
\usetikzlibrary{intersections}
\usepackage{epstopdf}
\usepackage{tikz}
\usetikzlibrary{arrows,automata,backgrounds,calendar,chains,matrix,mindmap,patterns,petri,shadows,shapes.geometric,shapes.misc,spy,trees,calc,3d}
\usepackage{hyperref}
\usepackage[linesnumbered,ruled,vlined]{algorithm2e}

\SetCommentSty{mycommfont}

\newtheorem{theorem}{Theorem}
\newtheorem{remark}{Remark}
\newtheorem{lemma}[theorem]{Lemma}

\newtheorem{proposition}{Proposition}%[theorem]

\newcommand{\R}{\mathbb{R}}
\newcommand{\C}{\mathbb{C}}
\renewcommand*{\d}[1]{\thinspace \dif#1}
\renewcommand*{\i}{\mathrm{i}}
\newcommand{\dif}{\mathrm{d}}

\newcommand{\x}{\mathbf{r}}
\newcommand{\y}{\mathbf{r'}}

\newcommand{\xz}{\x_{0}}
\newcommand{\zs}{{\boldsymbol{\xi}}}

\newcommand{\J}{\mathrm{\bf J}}
\newcommand{\Z}{\mathrm{\bf Z}}

\renewcommand{\e}{\mathrm{\bf e}}

\newcommand{\E}{\mathrm{\bf E}}

\newcommand{\curl}{\mathrm{\bf curl\,}}
\newcommand{\Om}{\mathbf{\Omega}}

\newcommand{\Oz}{D}

\newcommand{\n}{\textbf{n}}

\newcommand{\HVS}{V_{\xi}}

\newcommand\Green[2]{\mathcal{G}\!\left(#1,#2\right)}

\DeclareOldFontCommand{\bf}{\normalfont\bfseries}{\mathbf}

\begin{document}

\title{Near-Field Linear Sampling Method for Axisymmetric Eddy Current Tomography}

\author{H. HADDAR$^1$ and M. K. RIAHI$^{2,3}$\thanks{ Corresponding author:mohamed.riahi@ku.ac.ae}}
\address{$^1$ INRIA, Ecole Polytechnique (CMAP), Universit\'e Saclay Ile de France, Route de Saclay, 91128
Palaiseau, France
 \ead{haddar@cmap.polytchnique.fr}}
\address{$^2$ Department of Mathematics, Khalifa University of Sciences and Technology,\\
		P.O. Box 127788, Abu Dhabi, UAE. 
		              \ead{mohamed.riahi@ku.ac.ae}}
\address{$^3$ Emirates Nuclear Technology Center (ENTC), Khalifa University of Science and Technology, United Arab Emirates}

%\vspace{10pt}
%\begin{indented}
%\item[]August 2018
%\end{indented}

\begin{abstract}
This paper is concerned with Eddy-Current (EC) nondestructive testing of conductive materials and focuses, in particular, on extending the well-known Linear Sampling Method (LSM) to the case of EC equations.  We first present the theoretical foundation of the LSM in the present context and in the case of point sources. We then explain how this method can be adapted to a realistic setting of EC probes. In the case of identifying the shape of external deposits from impedance measurements taken from inside of the tube (steam generator), we show how the method can be applied to measurements obtained from a sweeping set of coils. Numerical experiments suggest that good results can be achieved using only a few coils and even in the limiting case of backscattering data. 
\end{abstract}

\vspace{2pc}

\submitto{\IP}
%\maketitle

% \tableofcontents

%~~~~~~~~~~~~~~~~~~
\section{Introduction}
%~~~~~~~~~~~~~~~~~~
We investigate the application of the so-called Linear Sampling Method (LSM) \cite{colton1997simple, bookCCH2016, audibert:hal-01422027} to identify conductive inclusions using eddy current probes. We specifically treat the axisymmetric case motivated by the non destructive testing  of tubes with axisymmetric coils. This setting corresponds for instance with the inspection of tubes used in steam generators of nuclear power plants for circulating hot water to produce steams (which is then used by a turbine to produce electricity) \cite{FORSTER197428}. One then is interested in detecting and evaluating the amount of magnetite deposits on the outer parts of the tube using coils that can be inserted inside the tubes \cite{PrusekPhD,jiang:hal-01220230,doi:10.1080/09349849809409630}. Given the high number of tubes, axisymmetric probes (usually referred to as SAX probes) are used in the inspection campaign that requires shutting down the power plant and evacuating the water circulating inside the steam generator.

Given the frequency range of the probes and the conductivity of the tube and the deposits, eddy current model is adapted to describe the measurements. In general, only back-scattering or nearly back-scattering data is measured. This is why mainly  optimization like methods have been used in the literature to tackle the inherent inverse problem \cite{haddar2017robust,riahi2016fast,jiang2011eddy,jiang2016identification, AGHJ2020}. 
In a perspective of studying other measurement settings or provide existing ones with good initial guess, we investigate here the performance of LSM for various measurements configurations. 

We refer to \cite{auld1999review,Blitz1997} and references therein for an overview of wider use of eddy currents in non destructive testings. Inverse shape problems for eddy current models have been studied in different contexts and using different methodologies in \cite{tamburrino2006fast,arnold2013unique,khan2008recursive,li2019learning} that may also be of interest for the application considered here.

We first analyze the idealistic case where coils can be approximated by point sources and measurements can be modeled as point-wise values of the scattered field. We formalize this setting in an axisymmetric configuration of the medium and for the eddy current model. Our analysis of the forward problem follows \cite{haddar:hal-01072091,MR1874449}. We then study the theoretical foundations of the LSM in the case where the deposit is characterized by a change in the conductivity with respect to a reference configuration. The measurements and sources are supposed to lie on an infinite line parallel to the axis of symmetry. Observing that the spread of the incident field is limited to a narrow area surrounding  the source location, the multi-static  measurement matrix is almost band limited (the entries outside some few diagonals may be negligible). This means that numerically, replacing the full matrix by an M-diagonal matrix may lead to an error that is of the order the noise level. This suggests that a good approximation of the measurement matrix may be achieved using only a few number, say $M$ of probes, (i.e. $(2M-1)$-diagonal matrix) that move together along the symmetry axis. At each position, the probes collect the measurements for a column of the matrix. 
This procedure corresponds with commonly used methods in practice for collecting the data. For the application mentioned above,  2 coils are used in general in SAX probes. 

Motivated by these considerations we numerically explore the outcome of the LSM when the full matrix is replaced by its $(2M-1)$-diagonal. We observe for instance that even for $M=1$ (i.e. back-scattering data), a reasonable localization and size estimate of the deposit are obtained.
After discussing the idealistic case of point sources, we extend the numerical algorithm to realistic modelling of the coils (as extended sources) and compute the measurements using the impedance model  \cite{auld1999review} (see also \cite{haddar2017robust}). The physical quantities and characteristic of the coils correspond with the one used in the application mentioned above. An extensive numerical discussion of the algorithm show the viability of our algorithm for this setting for providing quick and qualitatively accurate identification of the location, topology (number of connected components) and vertical dimensions of the deposit. 

The outline of the paper is as follows. We introduce in Section \ref{sec1} the physical model and recall some known results on the well posedness of the underlying PDE model. We present in Section \ref{sec2} the idealistic setting of the inverse problem corresponding with point sources. This requires in particular the introduction of the Green function for the axisymmetric problem. The theoretical foundations of the LSM are studied in Section \ref{sec3}. The last section is dedicated to a numerical validation of the LSM for various settings of the collected data. We start with the case of point sources then consider the case of more realistic models that correspond to SAX probes.
%~~~~~~~~~~~~~~~~~~
\section{Preambles}
%~~~~~~~~~~~~~~~~~~
\label{sec1}
Consider a simply-connected domain $\Om\subset\mathbb{R}^3$ with Lipschitz boundary. For a given applied divergence free electric current density $\J$, a frequency $\omega$, a magnetic permeability $\mu > 0$ and an electric conductivity $\sigma \ge 0$, the time harmonic eddy-current equations for the electric field $\E$ can be formulated as 
\begin{equation}\label{full-eddycurrent-equation-E}
\curl\left(\frac{1}{\mu}\curl \E\right) -i\omega\sigma \E =i\omega\J\quad\text{ in } \R^{3}
\end{equation}
supplemented with an appropriate gauge conditions for the divergence of the electric field \cite{rodriguez2010eddy,TR2014}.
Motivated by non destructive inspection of long tubes, we here consider the axi-symmetric configuration where all quantities  are invariant with respect to $\e_{\theta}$ where $(\e_{r},\e_{\theta},\e_{z})$  denotes the canonical basis of the cylindrical coordinate system. We further assume that $\J = J \,  \e_{\theta}$ with $J$ independent from the $\theta$ coordinate. Then the  electric field is azimuthal: $\E=E_{\theta}\e_{\theta}$ and satisfies 
\begin{equation}%\label{axissymmetric-eddycurrent-equation}
\dfrac{\partial}{\partial r}\left( \dfrac{1}{\mu r} \dfrac{\partial}{\partial r}\left(r E_{\theta}\right)\right) 
+ \dfrac{\partial}{\partial z} \left( \dfrac{1}{\mu}  \dfrac{\partial}{\partial z} E_{\theta}\right) + i\omega\sigma E_\theta = -i\omega J \quad \text{in } \R^{2}_{+}, 
\end{equation}
with $\R^{2}_{+} := \{ (r,z): r>0, z\in \R \}$. Due to symmetry, we also  have
  $$
  E_{\theta}|_{r=0} = 0,
$$ 
and we impose a decay condition 
$$
  E_{\theta} \rightarrow 0 \quad \text{as }  r^2+z^2 \rightarrow +\infty, 
$$ 
that models radiation condition at infinity. For a variational study of this problem we refer to \cite{haddar:hal-01072091} and we shall outline in the sequel the main result. Let $\xi >1$ be a fixed parameter. For functions of the variables $(r, z)$ we shall use in the following the short notation $\nabla:=(\partial_r,\partial_z)^T$ to refer to the gradient with respect to these two variables. We then define the weighted functional spaces
$$L^{2}_{\xi}(\Om):=\{v\, \text{measurable}\,|\, \sqrt{{r}\slash{(1+r^2)^\xi}}\,v\in L^2(\Om)\},$$
and
$$\HVS(\Om):=\{v\in L^{2}_{\xi}(\Om)\,|\, \nabla(rv)\slash\sqrt{r}\in L^2(\Om)\}.$$

The direct problem can then be stated as seeking $E_\theta = u \in \HVS(\R^2_+)$ such that 
\begin{equation}\label{axissymmetric-eddycurrent-equation}
\nabla\cdot\left(\dfrac{1}{\mu r} \nabla(r u) \right)  + i\omega\sigma u = -i\omega J  \quad\text{ in } \R^2_+,
\end{equation}
We observe that the boundary condition at $r=0$ and the decay at infinity are automatically included into the solution space. This problem is equivalent to the variational form: $u \in \HVS(\R^2_+)$
\begin{align}\label{eq:fv}
  \int_{\R^{2}_{+}} \frac{1}{\mu r} \nabla(ru)\cdot \nabla(r \bar{v}) \d r\d z 
    - \int_{\R^{2}_{+}} \i\omega \sigma u\bar{v} r \d r\d z
    =  \ell(v) 
    \quad \forall v\in \HVS(\R^2_+),
\end{align}
with 
$$
\ell(v) = \int_{\R^{2}_{+}} \i\omega J \bar{v} r\d r\d z
$$
We then have the following Theorem \cite{haddar:hal-01072091} (see also \cite{MR1874449}).
\begin{theorem}\label{prop:LaxMilgram}
  Assume that $\mu$ and $\sigma$ are in $L^\infty(\R^{2}_{+})$ and that
$\mu(r,z) \ge \mu_*> 0$ on  $\R^{2}_{+}$ and $\sigma=0$ for $r \ge r^*$ sufficiently large. Consider a source term $v \mapsto \ell(v)$ that is antilinear and continuous on $ \HVS(\R^2_+)$. Then, the variational problem \eqref{eq:fv} admits a unique solution $u \in \HVS(\R^2_+)$.
\end{theorem}

\section{Setting of the inverse problem for point sources}
\label{sec2}
We discuss in this section the inverse problem for an idealized configuration where the sources $J$  produced by coils can be considered as point sources. These sources are distributed on 
\begin{eqnarray*}
\Gamma_s &=&\left\{\x=(r,z)\in\R^2_+ \,|\, r=r_{s}\right\}
\end{eqnarray*}
that we shall assume to be located in a non conductive part (here $r_{s}$ is fixed). The reference domain (i.e. the domain that does not contains defaults) is described by a conductivity $\sigma_0$ and a permeability $\mu_0$ that satisfy the assumption of Theorem \ref{prop:LaxMilgram} and for simplicity we also assume that they are independent from the variable $z$.

Let us introduce the Green function associated with $\mu_0=1$ and $\sigma_0 =0$ which is the function $\Green{\cdot}{\xz}$ associated with a point source at $\xz = (r_0, z_0)$ satisfying
\begin{equation}\label{greensfunction}
\nabla \frac{1}{r} \cdot \nabla r \Green{\cdot}{\xz}= -\delta_{\xz}(\cdot) \quad     \text{ in } \R^2_+
\end{equation}
with homogeneous Dirichlet boundary conditions at $r=0$ and which is vanishing at $+\infty$. Noticing that (for $\theta$ being the angle of polar-coordinate system) the function $\Phi(x,y,z) := \Green{\x}{\xz} \sin \theta$ satisfies $\Delta \Phi =  - \sin \theta \delta_{\xz}$ in $\R^3$ with a decaying condition at infinity. It is, therefore, possible to derive an integral representation of $\Green{\cdot}{\xz}$ using the fundamental solution of the Laplace operator in $\R^3$, namely \cite{TR2014}
\begin{equation}\label{green3D-2D}
\Green{\x}{\xz} = \frac{1}{4\pi}\int_0^{2\pi}  \frac{r_0 \sin \theta'}{|r^2 + r_0^2 - 2rr_0 \sin \theta' + |z-z_0|^2|^{1/2}} \d\theta'
\end{equation}
{It has been shown} in \cite{Dini2004GREENFO} that this integral can be analytically identified as
{
\begin{equation}\Green{\x}{\xz} = \frac{1}{2\pi} \sqrt{\frac{r_0}{r}}Q_{1/2}\left(1+ \frac{|\x-\xz|^2}{2rr_0}\right)
\end{equation}
}
where $Q_{\nu}$ is the Legendre function of the second kind, satisfying
$$
(1-t^2) y''(t) - 2 ty'(t) + \nu (1+\nu) y(t) = 0.
$$
The asymptotic behavior of $Q_{\nu}(t) $ for large argument shows that $Q_{1/2}(t) \sim \frac{\pi}{\sqrt{32}} t^{-3/2}$, therefore 
$$
\Green{\x}{\xz} \sim \frac{r_0^2 r}{4 |\x|^3} \quad \text{as } |\x| \to \infty
$$
and 
$$
\Green{\x}{\xz} \sim \frac{r_0^2 r}{4 (r^2 + r_0^2 + |z-z_0^2|)^{3/2}} \quad \text{as } r \to 0. 
$$
This function allows us to define an incident field associated with a point source 
 $\xz\in\Gamma_{s}$ as the function $u^{0}(\cdot,\xz)$ satisfying 
\begin{equation}\label{u0eq}
    \nabla \frac{1}{\mu_0 r} \cdot \nabla r u^{0}(\cdot,\xz) + i\omega\sigma_{0} u^{0}(\cdot,\xz) = -\delta_{\xz}(\cdot) \quad \text{ in } \R^2_+
\end{equation}
with homogeneous Dirichlet boundary conditions at $r=0$ and which is vanishing at $+\infty$. This function can be constructed as 
\begin{equation}
\label{decomposition}
u^{0}(\cdot,\xz) = \mu_0(\xz) \Green{\cdot}{\xz} + \tilde u^{0}({\cdot},{\xz})
\end{equation}
where $\tilde u^{0}({\cdot},{\xz}) \in  \HVS(\R^2_+)$ and is the unique solution of  \eqref{eq:fv}
with $\mu=\mu_0$, $\sigma =\sigma_0$ and 
$$
\ell(v) = \int_{\R^2_+} \left(1-\frac{ \mu_0(\xz)}{\mu_0(\x)}   \right) \nabla r \Green{\x}{\xz} \cdot  \nabla r \bar v(\x) \d\x 
    + \int_{\R^{2}_{+}} \i\omega \sigma_0 \mu_0(\xz) \Green{\x}{\xz} \bar{v}(\x) r \d\x.
$$
We now consider the inverse problem configuration for the imaging of deposits inside the reference media defined by $\mu_0$ and $\sigma_0$. We assume that this deposit occupies a domain $D$ that lies in the region $r_s < r < r_* <\infty$.
\begin{remark}
It is worth noticing that studying the case of deposit in the region $D\subset \{r<r_s\}$ follows the same lines as in the case studied here. In the Engineering application we have in mind magnetite deposits occur in the outer shell side of the tube.
\end{remark}
    Let us denote by $\mu$ and $\sigma$ the functions defining the material properties of the domain with deposit. In particular $\mu=\mu_0$ and $\sigma = \sigma_0$ outside $D$. 

The probes are  point sources located at $\xz \in \Gamma_s$, that generate a field $u(\cdot,\xz)$ satisfying
\begin{equation}\label{u0eq}
    \nabla \frac{1}{\mu r} \cdot \nabla r u(\cdot,\xz) + i\omega\sigma u(\cdot,\xz) = -\delta_{\xz}(\cdot) \quad \text{ in } \R^2_+
\end{equation}
with homogeneous Dirichlet boundary conditions at $r=0$ and which is vanishing at $+\infty$. The field function can be defined as 
$$
u(\cdot,\xz) =  u^{0}({\cdot},{\xz}) + u^{s}({\cdot},{\xz})
$$
where $u^{s}({\cdot},{\xz})\in  \HVS(\R^2_+)$ is the scattered field that can be defined as   the unique solution of  \eqref{eq:fv}
with 
\begin{equation} \label{rhsgen}
\ell(v) = \int_{D} \left(\frac{1}{\mu_0(\x)} -\frac{1}{\mu(\x)}  \right) \nabla r u^{0}(\x,{\xz}) \cdot  \nabla r \bar v(\x) \d \x 
    - \int_{D} \i\omega (\sigma_0-\sigma) u^{0}({\x},{\xz}) \bar{v}(\x) r \d \x.
\end{equation}
Since the deposit $D$ does not intersect $\Gamma_s$, we clearly have that $u^{0}({\cdot},{\xz}) \in \HVS(D)$ and therefore this antilinear form is continuous on   $\HVS(\R^2_+)$. This guarantees by application of Theorem \ref{prop:LaxMilgram} the existence an uniqueness of $u^{s}({\cdot},{\xz})\in  \HVS(\R^2_+)$. Using a test function $\bar v(\y) = u^0(\y,\x)$ and integrating by parts leads to the following representation theorem.
\begin{proposition} \label{Intpresentation}
The scattered  field $u^{s}(\cdot,\xz)$ defined above satisfies for $r_0,r >0$ the reciprocity relation $r u^{s}(\x,\xz) = r_0 u^{s}(\xz,\x) $   as well as the integral representation
\begin{eqnarray}\label{us_formula}
\label{usImpedeq}
    r u^{s}(\x,\xz) &= & \int_{D} \left(\frac{1}{\mu_0(\y)} -\frac{1}{\mu(\y)}  \right) \nabla (r' u^{0}({\y},\x)) \cdot  \nabla ( r'  u(\y,\x_0) )\d \y 
    \\
    && +{i\omega}\int_{D} (\sigma-\sigma^{0}) u(\y,\xz) {u^{0}(\y,\x)} r' \,\d\y \notag.
\end{eqnarray}
\end{proposition}
\noindent It is worth noticing that the reciprocity relation is also satisfied by $\Green{\cdot}{\cdot}$ as well as $u^{0}(\cdot,\cdot) $.

The inverse problem we would like to address first is the problem of reconstructing $D$ from measurements of $ u^{s}(\x,{\xz})$ for all $\x$ and $\xz$ in $\Gamma_s$ using the so-called Linear Sampling Method. We shall later explain how this method can provide an inversion method for realistic setting (related to non destructive testing of conducting tubes) using few  eddy-current coils (and even back-scattering configurations). 

A key ingredient in the justification of the method is the following unique continuation argument associated with measurements. 

\begin{theorem} \label{uniquecontinuation}
Let $D$ be a bounded domain in the region $r > r_s$ with connected complement in $\R^2_+$. Let $u_1$ and $u_2$ in $\HVS(\R^2_+\setminus \overline D)$  and satisfy
\begin{equation}\label{u}
    \nabla \frac{1}{\mu_0 r} \cdot \nabla r u_i + i\omega\sigma_{0} u_i = 0  \quad \text{ in } \R^2_+ \setminus \overline D.
\end{equation}
If $u_1 = u_2$ on $\Gamma_s$, then  $u_1 = u_2$ in $\R^2_+ \setminus D$.
\end{theorem}
\begin{proof}
Denoting $\Omega := \{\x , \; 0 < r < r_s\}$, we observe that $v= u_1 -u_2 \in \HVS(\Omega) \cap H^1_0(\Omega)$ and verifies the variational formulation \eqref{eq:fv} with $\mu=\mu_0$, $\sigma=\sigma_0$, $\ell =0$  and   $\HVS(\R^2_+)$ replaced with $\HVS(\Omega) \cap H^1_0(\Omega)$. The  coercivity of the associated sesquilinear form in $\HVS(\Omega) \cap H^1_0(\Omega)$ implies that $u_1 = u_2$ in $\Omega$. We then conclude that  $u_1 = u_2$ in $\R^2_+ \setminus D$ using classical unique continuation argument for elliptic second order operators with piecewise constant coefficients (with  Lipschitz interfaces). 
\end{proof} 
\section{Foundations of the linear sampling method and algorithm}
 \label{sec3}
The extension of the linear sampling method to the current setting does not raise major difficulties or differences with respect to the classical setting of method for inverse scattering problem \cite{colton1997simple,bookCCH2016,audibert:hal-01422027}. This is why we shall give in the following only the outline of this method and eventually the key points of the proofs. The assumptions made in the previous section for $\sigma$, $\sigma_0$, $\mu$ and $\mu_0$  so that the forward problems defining $u^0(\cdot, \x_0)$ and $u(\cdot, \x_0)$ are well posed are assumed to hold true and we will not further indicate that in the subsequent theorems or results.

 To simplify the technical details we further assume here that $\mu=\mu_0$ and therefore the deposit is characterized only by variation of the conductivity value $\sigma$.

 We introduce the measurement operator $\mathcal{Z}$ as follows: 
\begin{eqnarray}\label{vigeq}
    \mathcal{Z}&:& L^{2}(\Gamma_s)\longrightarrow L^{2}(\Gamma_s)\\
    && g \quad\longmapsto\quad (\mathcal{Z}g)(\x) := \int_{\Gamma_s}  u^{s}(\x,{\xz}) g(\xz)\d s(\xz) \quad \x \in \Gamma_s
\end{eqnarray}
The linearity of $u^{s}$ with respect to $u^{0}$ shows that $\mathcal{Z}g$ corresponds with the trace of  $u^{s}_g$  on $\Gamma_s$ where 
 $u^{s}_g\in  \HVS(\R^2_+)$ is the unique solution of  \eqref{eq:fv}
with source term $\ell$ defined by \eqref{rhsgen} where $u^{0}$ is replaced by the  single layer potential  
\begin{equation}\label{iHWF}
    v^{0}_{g}(\x) = \int_{\Gamma_{s}} u^{0}(\x,\xz)g(\xz)\,\d s(\xz).
\end{equation}
We then have a natural decomposition of the measurement operator as 
$$
\mathcal{Z} = \mathcal{G} \circ \mathcal{S}
$$
where the  operator $ \mathcal{S} : L^2(\Gamma_s) \to L^2(D)$ is defined by $\mathcal{S} g := v^{0}_{g}|_D$ and where the solution operator $\mathcal{G} : L^2(D) \to L^2(\Gamma_s)$ is defined by $\mathcal{G}(v) = w|_{\Gamma_s}$ with  $w \in  \HVS(\R^2_+)$ being the unique solution of  \eqref{eq:fv}
with source term $\ell$ defined by \eqref{rhsgen} replacing $u^{0}$ by the function $v \in L^2(D)$.
We first prove the following important properties of the operator $\mathcal{S}$.
\begin{lemma} \label{Stheorem}
The operator $ \mathcal{S} : L^2(\Gamma_s) \to L^2(D)$ is injective. The range of this operator is dense in 
\begin{equation}\label{DensityLemma}
V_g(D) := \left\{ v \in L^{2}(\Oz), \quad\,\nabla \frac{1}{\mu_0 r} \cdot \nabla r v + i\omega \sigma_0 v  =0, \text{ in } \Oz \right\}.
\end{equation}
\end{lemma}
\begin{proof}
We first prove the injectivity. We observe that according to \eqref{green3D-2D} and the preceding discussion,
if we set $g\in L^2(\Gamma_s)$
$$
SL(g)  (\x) := \int_{\Gamma_{s}} \mathcal{G}(\x,\xz)g(\xz)\,\d s(\xz).
$$
then 
\begin{equation}
SL(g)  (\x)\sin(\theta) = \int_{\Gamma_s \times [0, 2\pi]} \frac{\tilde g (\mathbf{x}_0)}{4\pi|\mathbf{x} -\mathbf{x}_0|} ds(\bf{x}_0)
\end{equation}
where $\mathbf{x} = (r\cos\theta, r\sin\theta,z)$, $\mathbf{x}_0= (r_0\cos\theta', r_0\sin\theta',z_0)$ and $\tilde g (\mathbf{x}_0) = g(\x_0) \sin \theta'$. This remark allows us to deduce that $SL(g) $ has the same continuity properties across $\Gamma_s$ as the single layer potential associated with Laplace operator in $\R^3$ \cite{McLean}. In particular $SL(g)$ is continuous across $\Gamma_s$ while its normal derivative  (derivative with respect to $r$ here) has a jump across $\Gamma_s$ equals to $- g$. Given the decomposition \eqref{decomposition} of $u_0(\cdot, \x_0)$ and the fact that $\tilde u_0(\cdot, \x_0)$ is regular in the neighborhood of $\Gamma_s$, one deduces that $v^{0}_{g}$ has the same continuity properties as $SL(g)$ across $\Gamma_s$. Now assume that  $\mathcal{S} g = 0$. This implies that $ v^{0}_{g}=0$ in $D$. Since $v^{0}_{g}$ satisfies 
\begin{equation}\label{v0g}
    \nabla \frac{1}{\mu_0 r} \cdot \nabla rv^{0}_{g}  + i\omega\sigma_{0} v^{0}_{g} = 0  \quad \text{ in } \R^2_+ \setminus \Gamma_s.
\end{equation}
we deduce, using classical unique continuation arguments for elliptic second order operators with piecewise constant coefficients,  that $v^{0}_{g} = 0$ for $r > r_s$. This implies in particular that $v^{0}_{g} = 0$ on $\Gamma_s$. Following the same reasoning as in the proof of Theorem \ref{uniquecontinuation}, we then deduce that $v^{0}_{g} = 0$ for $0<r < r_s$ since  it satisfies the homogeneous Dirichlet problem in the region $0<r < r_s$. The jump properties of the normal derivative of across $\Gamma_s$ finally implies that $g=0$ which conclude the proof of injectivity. For the denseness of the range, it is sufficient to prove that the adjoint operator $\mathcal{S}^*$ is injective on $V_g(D)$. It is natural to consider the adjoint with respect to the duality product $\int_D f(\x) g(\y) r d\x$, in which case
$$
\mathcal{S}^*   v  (\x_0) = \int_D u^0(\x, \x_0) v(\x) r d\x 
$$ 
Denote by $r_0 w(\x_0)$ the function defined by the right hand side for $\x_0 \in \R_+^2$. We then have $w\in \HVS(\R^2_+)$ and satisfies (similarly to the statement in Proposition \ref{Intpresentation}
\begin{equation}\label{w0g}
    \nabla \frac{1}{\mu_0 r} \cdot \nabla rw  + i\omega\sigma_{0} w = - \overline v|_D \quad \text{ in } \R^2_+ 
\end{equation}
with Dirichlet boundary condition at $r=0$. Assuming that $\mathcal{S}^*   v =0$ implies that $w$ vanishes on $\Gamma_s$ and by Theorem \ref{uniquecontinuation}, $w=0$ in $\R_+^2 \setminus D$. Multiplying \eqref{w0g} with $r v$ and integrating over $D$ implies 
after applying the Green Theorem twice (this can be justified using a density argument)
$$
\int_D |v|^2 rd\x = \int_D w (\nabla \frac{1}{\mu_0 r} \cdot \nabla rv  + i\omega\sigma_{0} v) r d\x = 0 
$$
which implies $v=0$ and concludes the proof.
\end{proof}

The study of LSM for penetrable media requires the analysis of so-called interior transmission problems. In our context this problem can be formulated as: seek $(w, v)\in H^2(\Oz)\times L^2(\Oz)$ such that
\begin{equation}\label{ITP}
\tag{ITP}
\begin{cases}
 \nabla \frac{1}{\mu_0 r} \cdot \nabla r w + i\omega \sigma w = - i\omega(\sigma-\sigma_0) v & \text{ in } \Oz\\
\nabla \frac{1}{\mu_0 r} \cdot \nabla r v + i\omega \sigma_0 v =0  & \text{ in } \Oz\\
w = f  & \text{ on } \partial\Oz\\
\dfrac{\partial w}{\partial \n} = h  & \text{ on } \partial\Oz
\end{cases}
\end{equation}
where $f\in H^{3/2}(\partial\Oz)$ and $h\in H^{1/2}(\partial\Oz)$. The well posedness of this problem is an essential assumption for the following arguments justifying the Linear Sampling method. Let us remark that since we assume that $D$ is bounded and does not touch $r=0$,  the study of this problem follows the same lines as the study of the well posedness of this problem where the operator $\nabla \frac{1}{\mu_0 r} \cdot \nabla r$ is replaced with the operator  $\Delta$. For instance this problem is of Fredholm type as long as $\sigma-\sigma_0$ is positive definite or negative definite in a neighborhood of the boundary of $D$. The uniqueness of solution can be established for all $\omega$ if one assumes that $\sigma_0 = 0$ and $\sigma$ is positive (on some sub-domain of) $D$. The study of \ref{ITP} in the case where $D$ touches $r=0$ or in the case where the domain is unbounded may require additional arguments than those classically used in \cite{bookCCH2016}. A detailed discussion of these issues are out of the scope of this work and may be the subject of a future work. For the present case we restrict ourselves to assuming that $\omega$, $\mu_0$, $\sigma$ and $\sigma_0$ are such that \eqref{ITP} admits a unique solution.

For a given point $\zs$ in $\R_+^2$, we denote by $\phi_\zs \in L^2(\Gamma_s)$ the function defined by
$$
\phi_\zs (\x) := u^0(\x, \zs)  \quad \x \in \Gamma_s
$$
\begin{theorem} \label{Gtheorem}
Assume that \eqref{ITP} is well posed. Then the operator $\mathcal{G}$ is injective on $V_g(D)$. Moreover, the equation  $\mathcal{G}(v) =  \phi_\zs$ admits a solution $v\in V_g(D)$ if and only if $\zs \in D$.
\end{theorem}
\begin{proof}
Using the unique continuation argument of Theorem \ref{uniquecontinuation} and the fact that $u^0(\cdot, \zs) \notin H^1(\R^2_+\setminus \overline D)$  if $\zs \notin D$, the proof of this theorem follows the same arguments as in \cite{bookCCH2016}. We here outline the main arguments. If $\mathcal{G}(v) = 0$ then by Theorem \ref{uniquecontinuation}, this is equivalent to the existence of a solution to the homogeneous interior transmission problem \eqref{ITP}. This implies that $v =0$. If $\zs \in D$, the existence of a solution to $\mathcal{G}(v) =  \phi_\zs$ is ensured by the existence of a solution to \eqref{ITP} with
$$
f :=  u^0(\x, \zs)|_{\Gamma_s} \quad h:= \dfrac{\partial u^0(\x, \zs)}{\partial \n} |_{\Gamma_s}.
$$
If $\zs \notin D$ then \ref{uniquecontinuation} ensures that $w=u^0(\cdot, \zs)$ in $\R^2_+\setminus \overline D$ which contradicts the fact that $u^0(\cdot, \zs) \notin H^1(\R^2_+\setminus \overline D)$.
\end{proof}
As a corollary of Theorem \ref{Stheorem} and Theorem \ref{Gtheorem}, we obtain the following Theorem.
\begin{theorem}
Assume that \eqref{ITP} is well posed. then the operator $\mathcal{Z}$ is injective with dense range. Moreover, the following holds.
\begin{itemize}
\item If $\zs \in D$, then there exists a sequence  $g^\epsilon$ such that $\| \mathcal{Z} g^\epsilon -  \phi_\zs \|_{L^2(\Gamma_s)} \to 0 $ as $\epsilon \to 0$ and $\lim_{\epsilon \to 0} \| \mathcal{S} g^\epsilon \|_{L^2(D)} < + \infty$.
\item  If $\zs \notin D$ then for any sequence $g_\epsilon$ such that $\| \mathcal{Z} g^\epsilon -  \phi_\zs \|_{L^2(\Gamma_s)} \to 0 $ as $\epsilon \to 0$, $\lim_{\epsilon \to 0} \| \mathcal{S} g^\epsilon \|_{L^2(D)} = + \infty$.
\end{itemize}
\end{theorem}
The proof of this theorem follows exactly the same lines as in \cite[Chapter 2]{bookCCH2016}. We just emphasize that the  denseness of the range of $\mathcal{Z}$ is equivalent to its injectivity thanks to the reciprocity relation in Proposition \ref{Intpresentation}.

This theorem suggests to construct nearby solutions $\mathcal{Z} g^\epsilon_\zs \simeq \phi_\zs$ and construct an indicator for $D$ based on some norm related to $g^\epsilon_\zs$. In the literature, numerical experiments proved that an appropriate implementation would consist in applying a Tikhonov regularization by solving
$$
\epsilon \,  g^\epsilon_\zs +  \mathcal{Z}^*\mathcal{Z} g^\epsilon_\zs = \mathcal{Z}^* \phi_\zs.
$$
This choice is motivated by the denseness of the range of $\mathcal{Z}$ that ensures the construction of a nearby solution. The numerical details relate to this procedure is explained in the following section as well as the adaptations made for realistic configurations.

Consider a finite number of sources $\x_i$, $i=1,\ldots N$, equidistantly spaced on $\Gamma_s$. We use the finite element package FreeFem++\cite{ffpp} to generate the incident fields and scattered fields associated with these point sources and a given deposit characterized by its shape and conductivity. The problem is discretized using P1-Lagrange elements (We also use the option adaptmesh in FreeFem++ to increase the accuracy of the computations). We then numerically evaluate the data matrix 
$$
\mathcal{Z}_{i,j} := u^s(\x_i, \x_j).
$$
This data is then corrupted with a random noise of level $\delta$ as $\mathcal{Z}_{i,j} = (1 + \delta_{i,j}) \mathcal{Z}_{i,j}$ where $\delta_{i,j}$ are complex numbers with real and imaginary parts randomly and uniformly chosen in the interval $[ -\delta, \delta]$. 
The inversion algorithm takes this data together with the level noise $\delta$ as entries, then builds an indicator function for the deposit as follows: 

\begin{itemize}
 \item Consider a uniform sampling of the probed region (outside the tube), which is a rectangle that contains the deposit and with a  height at least equal to  the distance between the farthest point sources. Consider also a sampling point $\zs_\ell$ in this region. We evaluate the
the right hand side $\phi_\ell \in \C^N$ as
$$
\phi_\ell(i) := u^0(\x_i, \zs_\ell),
$$
\item We compute the function $g_\ell \in \C^N$ solution of the regularized equation, namely
 \begin{equation}\label{reguSystem}
\epsilon_\ell \,  g_\ell +  \mathcal{Z}^*\mathcal{Z} g_\ell = \mathcal{Z}^* \phi_\ell.
\end{equation}
The regularization parameter $\epsilon_\ell$ is evaluated using the Morozov's principle, i.e. it is selected so that it ensures the equality
$$
\| \mathcal{Z} g_\ell - \phi_\ell\| = \delta \| g_\ell\|.
$$
\end{itemize}

We finally plot the indicator function $\ell \mapsto 1/ \| g_\ell\|$. This procedure is the usually adopted method to build the criterion for sampling methods \cite{colton1997simple}.

\medskip

\noindent {\bf Discussion of the case of limited number of receivers}.

\medskip

The previous procedure requires to have $N$ receivers (coils) which may be for practical applications  hard to set up if $N$ is large. In general only a limited number of coils $M \ll N$ are available. For the industrial experiment mentioned earlier, these $M$-probes move together along the tube axis. Consequently, one has access only to $2M-1$ sized band-diagonal of the matrix $\mathcal{Z}$. With regards to the spreading of the solution around the source location (Figure \ref{spreadsol}), we see that $u^s(\x_i, \x_j) \simeq 0 $ if the distance between the points  $\x_i$ and $\x_j$ is sufficiently large. Therefore, approximating $\mathcal{Z}$ using the $2M-1$ diagonals can be reasonable if $M$ is sufficiently large. In practice $M$ can be reduced to $M=1$ (back-scattering data) or $M=2$ (using 2 coils and symmetries). This is why we also experiment in the following  small values of $M$, where surprisingly good results are obtained.

Our inversion algorithm for the case of data provided with $M$ moving coils ($M<N$) is the same as for having $N$ coils replacing the full matrix $\mathcal{Z}$ by the matrix $\mathcal{Z}^M$ that is obtained from $\mathcal{Z}$ by putting $0$ outside the diagonal central band of size $M$. More specifically, the matrix $\mathcal{Z}^M$ is defined by
$$
\mathcal{Z}^M_{i,j} = \mathcal{Z}_{i,j} \mbox{ if } |i-j| \le M \mbox{ and } \mathcal{Z}^M_{i,j} =0 \mbox{ if not.}
$$ 
Figure \ref{full-prob} gives an illustration of $\mathcal{Z}^M$ for $M=4$ and $N=32$.
 \begin{figure}[htbp]
\centering
\includegraphics[width=6cm]{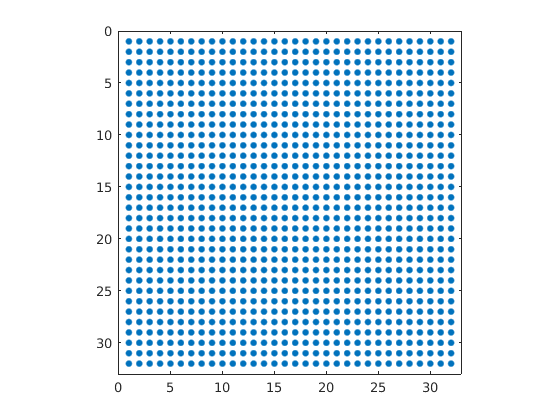}
\includegraphics[width=6cm]{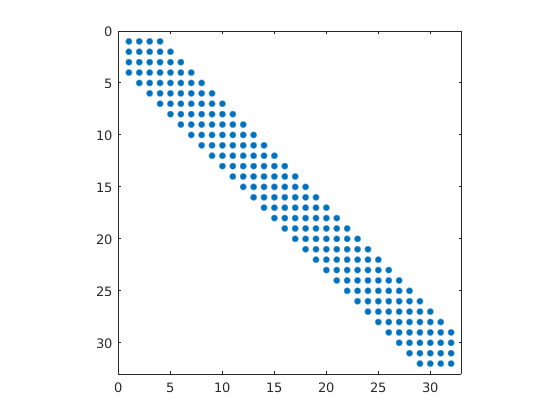}
\caption{Profile of the non zero entries of the matrix $\mathcal{Z}^M$ (right) with $M=4$ compared to the matrix $\mathcal{Z}$ (left).}
\label{full-prob}
\end{figure}

\section{Numerical experiments and validation}
%~~~~~~~~~~~~~~~~~~

\subsection{Description of the targeted application}
The numerical experiments conducted in this section are motivated by the industrial application of identifying the shape of magnetite deposits on the external surface of tubes inside steam generators from measurements of eddy currents associated with a co-axial coil inserted inside the tube. We refer to \cite{jiang2016identification,haddar2017robust} for a  description of the industrial context and the experimental setting.  Figure~\ref{fig-geo} provides a sketch illustrating the experiment. 
The conductive parts are formed by the tube and the deposits. 

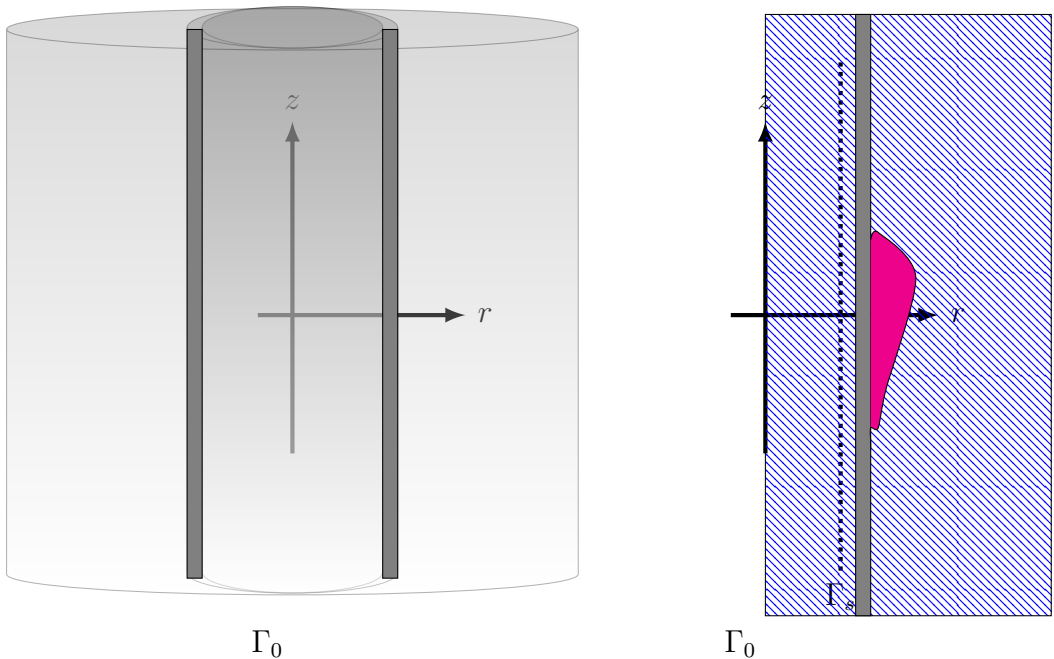
\begin{figure}[htbp]\centering
\begin{minipage}[t]{0.4\linewidth}
\begin{tikzpicture}
\begin{scope}[scale=0.46]
\draw[-latex,ultra thick] (-1,0)--(5,0) node[right]{$r$};
\draw[-latex,ultra thick] (0,-4)--(0,5.6) node[above]{$z$};
\node [draw, cylinder, shape aspect=2, rotate=90, minimum height=7.8cm, minimum width=2.4cm,shading=axis,opacity=0.34] (c1) {};
\node [draw, cylinder, shape aspect=2, rotate=90, minimum height=7.8cm, minimum width=2.8cm,shading=axis,opacity=0.32] (c2) {};
\node [draw, cylinder, shape aspect=2, rotate=90, minimum height=7.8cm, minimum width=7.6cm,shading=axis,opacity=0.3] (c2) {};
\end{scope}
%\draw[.,ultra thick] (3.8,-4) node[below]{$\Gamma_{\star}$} --(3.8,4);
\draw [fill=gray] (1.2,-3.5) rectangle (1.4,3.8);
\draw [fill=gray] (-1.2,-3.5) rectangle (-1.4,3.8);
\node[above right=10pt of {(-1,-5)}, outer sep=2pt,fill=white] {$\Gamma_{0}$};
\end{tikzpicture}\end{minipage} \hfill
\begin{minipage}[t]{0.4\linewidth}
\begin{tikzpicture}
\begin{scope}[scale=0.46]
\draw[-latex,ultra thick] (-1,0)--(5,0) node[right]{$r$};
\draw[-latex,ultra thick] (0,-4)--(0,5.6) node[above]{$z$};
\end{scope}
\draw[dotted,ultra thick] (1,-3.4)  node[below]{$\Gamma_{s}$}--(1,3.4);
%\draw[.,ultra thick] (3.8,-4) node[below]{$\Gamma_{\star}$} --(3.8,4);
\draw[pattern=north west lines, pattern color=blue] (0,-4) rectangle (3.8,4);
\draw [fill=gray] (1.2,-4) rectangle (1.4,4);
\node[above right=10pt of {(-1,-5)}, outer sep=2pt,fill=white] {$\Gamma_{0}$};
\begin{scope}
\draw[fill=magenta] plot [smooth] coordinates {(1.4,1) (1.5,1.1) (2,0.5) (1.6,-1) (1.5,-1.5) (1.4,-1.5) (1.4,-1.7) (1.4,-1.8)};
\end{scope}
\end{tikzpicture}\end{minipage}
\caption{Geometry configuration of the tube and its surrounding. Axisymmetric cut for the geometry shows: A deposit free configuration (left) and a deposit at the external part of the tube (right) with the probe positions on $\Gamma_s$.}
\label{fig-geo}
\end{figure}

 The setting for the experiment is inspired from \cite{jiang2016identification} which yields the spacing and dimensions of the coils as indicated in Figure \ref{fig-param}.

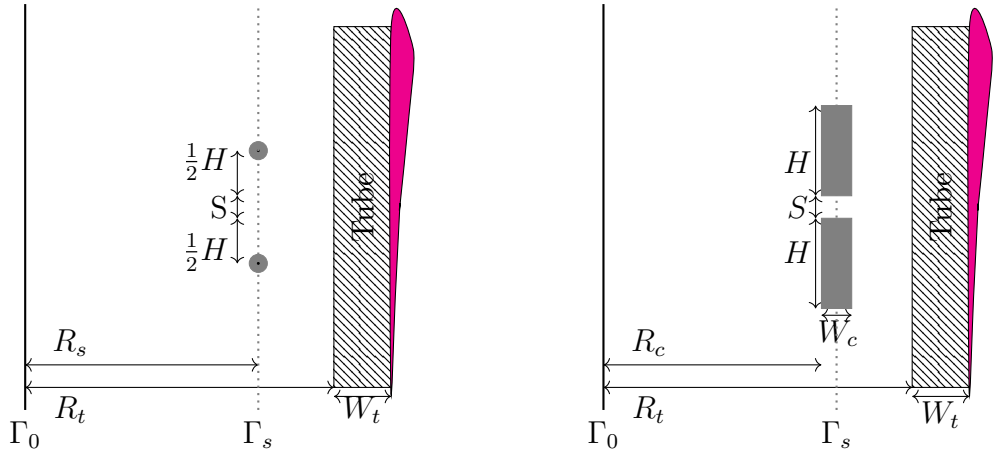
\begin{figure}[htbp]
\centering
\begin{minipage}[t]{0.49\linewidth}
\begin{tikzpicture}[scale=.6] 
\draw [<->] (7.7,-0.25)--(7.7,-1.25); \node (H) at (7.,1) {$\frac 1 2 H$};
\draw [<->] (7.7,-0.25)--(7.7, 0.25); \node (W) at (7.3,0) {S};
\draw [<->] (7.7, 1.25)--(7.7, 0.25); \node (H) at (7.,-1) {$\frac 1 2 H$};
\draw [<->] (9.84,-4.2)--( 11.11,-4.2); \node (W) at (10.47,-4.5) {$W_t$};
\node[rotate=90] (Tube) at (10.47,0) {Tube};
\draw [<->] (3,-3.5)--(8.165 ,-3.5); \node (R1) at (4,-3) {$R_s$};
\draw [<->] (3,-4)--(9.84 ,-4); \node (R2) at (4,-4.5) {$R_t$};
\filldraw[gray]  (8.165,-1.25) circle (0.2cm);
\filldraw[black]  (8.165,-1.25) circle (0.02cm);
\filldraw[gray]  (8.165,1.25) circle (0.2cm);
\filldraw[black]  (8.165,1.25) circle (0.02cm);
\filldraw[pattern=north west lines]   (9.84,-4) rectangle (11.11,4);
\draw[dotted,gray,thick] (8.165,-4.5)  node[below,black]{$\Gamma_{s}$}--(8.165,4.5);
\draw[thick] (3,-4.5)  node[below]{$\Gamma_{0}$}--(3,4.5);
\draw[fill=magenta] plot [smooth cycle] coordinates {(11.11,4) (11.6,3.5) (11.5,2) (11.3,0)
(11.3,0) (11.2,-2) (11.1,-4)};
\end{tikzpicture}
\end{minipage}%\hfill 
\begin{minipage}[t]{0.49\linewidth}
\begin{tikzpicture}[scale=.6] 
\draw [<->] (7.7,-2.25)--(7.7,-0.25); \node (H) at (7.3,1) {$   H$};
\draw [<->] (7.7,-0.25)--(7.7, 0.25); \node (W) at (7.3,0) {$S$};
\draw [<->] (7.7, 0.25)--(7.7, 2.25); \node (H) at (7.3,-1) {$H$};
\draw [<->] (7.83,-2.4)--( 8.5,-2.4); \node (W) at (8.2,-2.8) {$W_c$};
\draw [<->] (9.84,-4.2)--( 11.11,-4.2); \node (W) at (10.47,-4.6) {$W_t$};
\node[rotate=90] (Tube) at (10.47,0) {Tube};
\draw [<->] (3,-3.5)--(7.83 ,-3.5); \node (R1) at (4,-3) {$R_c$};
\draw [<->] (3,-4)--(9.84 ,-4); \node (R2) at (4,-4.5) {$R_t$};
\filldraw[gray]  (7.83,-2.25) rectangle (8.5,-0.25);
\filldraw[gray]  (7.83, 2.25) rectangle (8.5, 0.25);
\filldraw[pattern=north west lines]   (9.84,-4) rectangle (11.11,4);
\draw[dotted,gray,thick] (8.165,-4.5)  node[below,black]{$\Gamma_{s}$}--(8.165,4.5);
\draw[thick] (3,-4.5)  node[below]{$\Gamma_{0}$}--(3,4.5);
\draw[fill=magenta] plot [smooth cycle] coordinates {(11.11,4) (11.6,3.5) (11.5,2) (11.3,0)
(11.3,0) (11.2,-2) (11.1,-4)};
\end{tikzpicture}
\end{minipage}
\caption{Sketch of the probing configuration and parameters for the sources. Right: height of the coil $H=2$mm, width of the coil $W_c=0.67$mm, separation between coils $S=0.5$mm, radius of the internal boundary of the coil $R_c=7.83$mm, internal radius of the tube $R_t=9.84$mm.  
Left: corresponding parametrisation for idealistic case of point sources. The sources  are located at $\Gamma_s$ with radius $R_s=8.165$mm which is equal to the mid radius of the coils. Successive point sources are separated by $H+S=2.5$mm.}
\label{fig-param}
\end{figure}

Figure \ref{spreadsol} illustrates an incident field $u^0$ and a scattered field $u^s$ in the case of point sources and the case of rectangular coil described in Figure \ref{fig-param}. The scattered field correspond with two small deposits indicated by a solid line. We observe in particular the spreading of the incident field is very limited and most of the energy is captured by the two or three neighboring coils. 

\begin{figure}[htbp]\centering
\begin{tabular}{cccc}
\includegraphics[width=3.65cm, height=9cm]{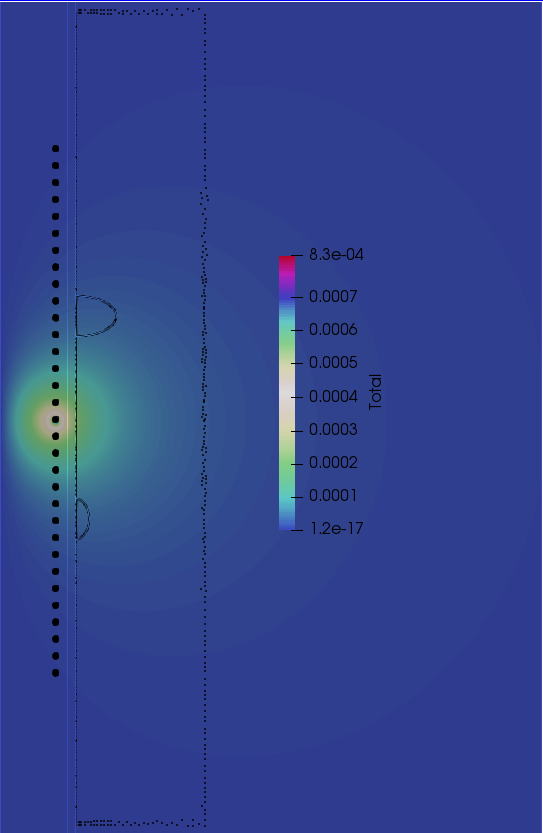}&
\includegraphics[width=3.65cm, height=9cm]{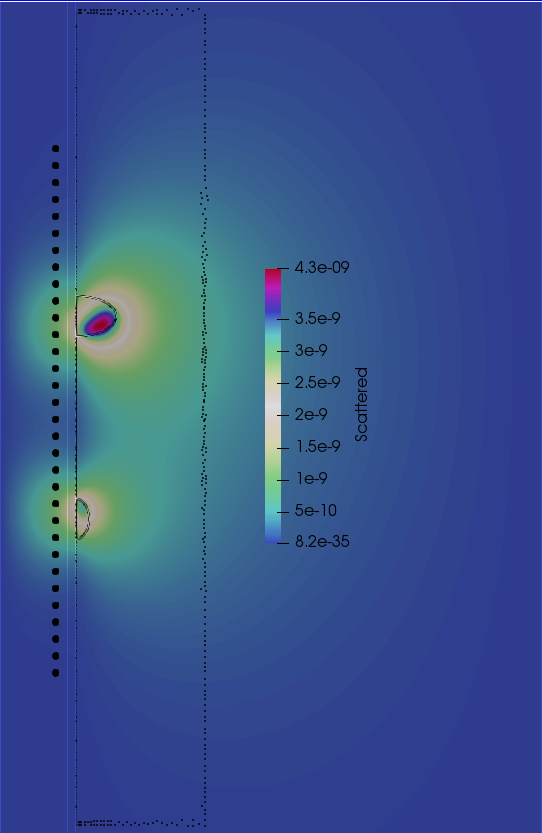}&
\includegraphics[width=3.65cm, height=9cm]{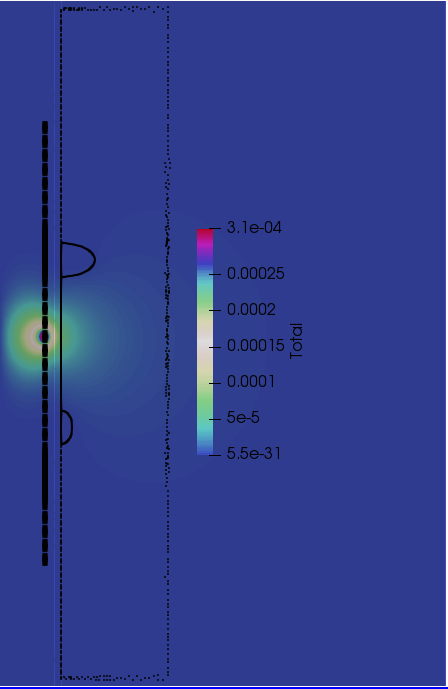}&
\includegraphics[width=3.65cm, height=9cm]{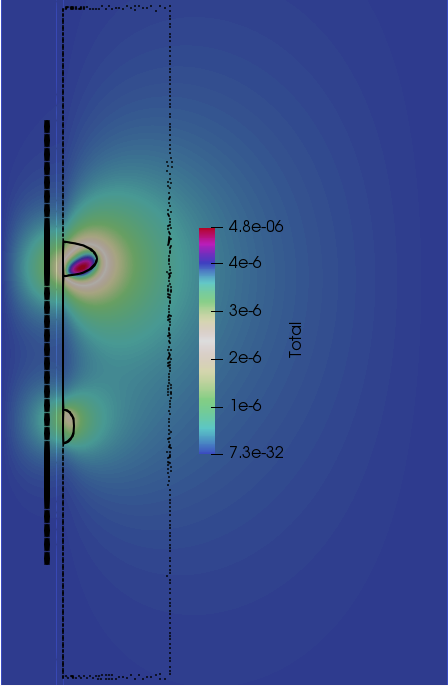}
\\ (a) & (b) & (c) & (d)
\end{tabular}
\caption{Eddy-current finite element solution of the total field with frequency used $\omega=200\pi$, in presence of two conductive deposits with semi-disc shape distanced by about $24.5$mm. From left to right: (a) total field using Green's function, (b) scattered field using point sources, (c) total field spread using probing coils and (d) scattered field using probing coils.}
\label{spreadsol}
\end{figure}

In the case of realistic experiments, only a finite number of coils are used. Moreover, the field generated by the coil is slightly different from the one generated by a point source. With this respect, we shall discuss first a validation of the model problem treated in the theoretical part when the point source approximation holds true. We then discuss how we can extend the algorithm to cases where the point source approximation does not hold.

\subsection{Discussion of the inversion results in the case of point sources}
 
We hereafter present some numerical experiments in the case of point sources. We use physical parameters compatible with the realistic configurations of the tube and deposits (see Figure \ref{fig-param} for geometric parameters details and Table~\ref{PhysicalProperty} for physical details of all materials). We consider only the case of simply connected deposit while other configurations can be seen below for the case of coils.

\subsubsection*{Example of reconstructions obtained with $\mathcal{Z}$ for different values of $N$.}

In this series of experiments we vary the number of sources (keeping the same vertical spacing) and use the full matrix to image the domain containing the deposit. It is observed in Figure \ref{dtype_1_PSRC} that a good accuracy is obtained with sufficiently large number of sources that covers an aperture larger than the deposit elongation. When the number of coils becomes small  (4 on the left), the results indicates that a good localization maybe achieved if the sources are facing the deposit.  This motivates the use of small number of coils sliding along the $z$ direction and covering a large aperture as discussed in the following experiment.

\begin{figure}[htbp]
\includegraphics[width=3cm, height=9cm]{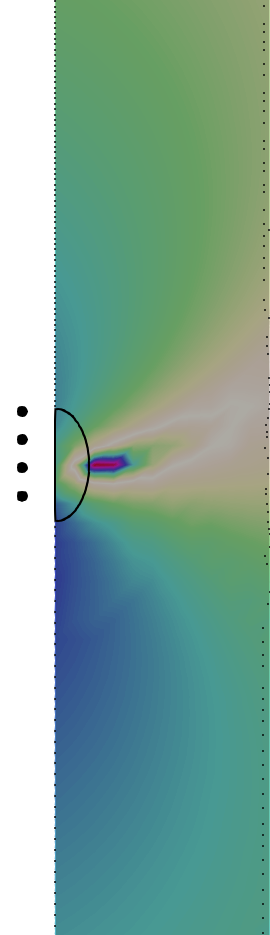}\hfill
\includegraphics[width=3cm, height=9cm]{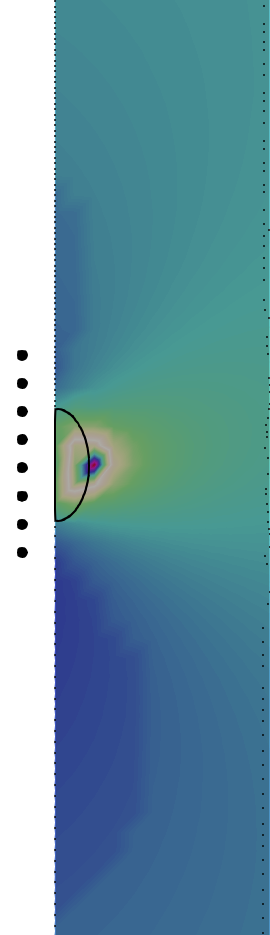}\hfill
\includegraphics[width=3cm, height=9cm]{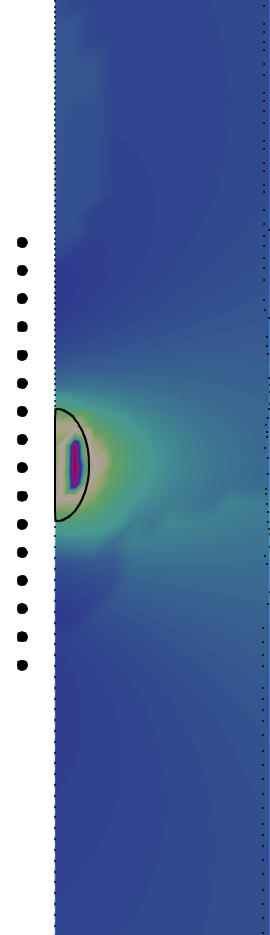}
\caption{Reconstructions obtained for different number of point sources $N$. From left to right, $N=4, \, 8, \, 16$. The location of the point sources is indicated with black small circles and the exact shape of the deposit is indicated  using a solid line and it represents a semi-disc with an elliptical shape of radius 3mm in the $y$-direction and 5mm in the $z$-direction.}
\label{dtype_1_PSRC}
\end{figure}

\subsubsection*{Examples of reconstructions obtained with $\mathcal{Z}^M$ for different values of $M$.}

We here consider the same configuration as previously, fix $N=2^5$ (this refers here to the number of positions that one point source may take) and vary the value of $M$ which indicates the number of point sources that are sliding along the $z$-axis. The obtained results are illustrated by Figure~\ref{dtype_1_PSRC_prob}. We observe in particular that even with back-scattering data, a good localization along the z-axis is obtained. Increasing the number of point sources improves the accuracy along the axis orthogonal to the source location. 
\begin{figure}[htbp]
\includegraphics[width=3cm, height=9cm]{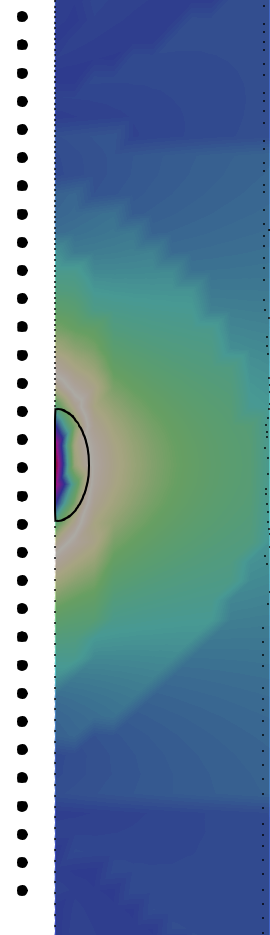}\hfill
\includegraphics[width=3cm, height=9cm]{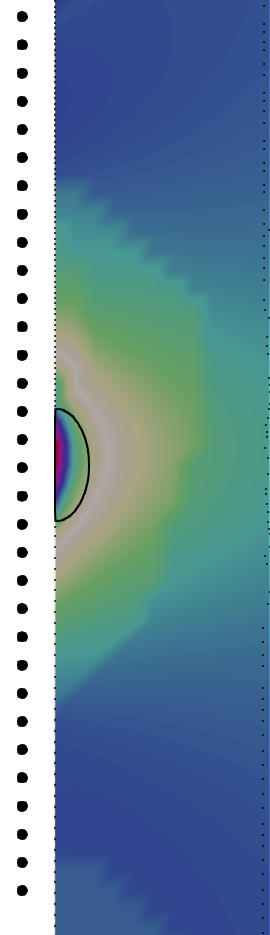}\hfill
\includegraphics[width=3cm, height=9cm]{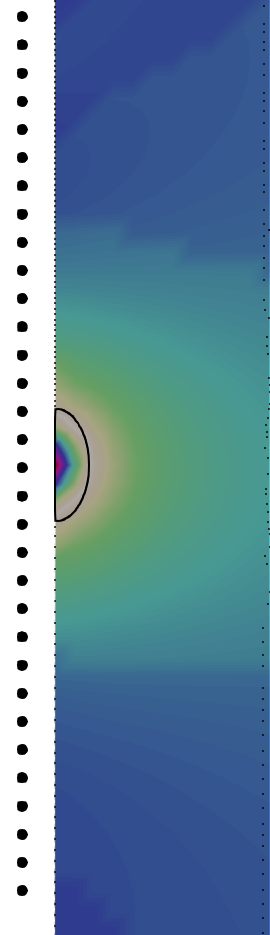}
\caption{Reconstructions obtained with $\mathcal{Z}^M$ for different number of point sources $M$ sliding along the $z-axis$. From left to right, $M=1, \, 2, \, 8$. The location that can be taken by a  point source is indicated with black discs ($N=2^5$ and the exact shape of the deposit is indicated with a solid line and it represents a semi-disc with an elliptical shape of radius 3mm in the $y$-direction and 5mm in the $z$-direction.}\label{dtype_1_PSRC_prob}
\end{figure}

\subsection{Extension to the case of realistic coils}
In real experiments the coils are better represented as volumetric sources $J$ with constant intensity $I$ in the region of the coil and vanishing outside. The coil region is a rectangle characterized  by a center $\x_j \in \Gamma_s$ and  dimensions $H$ and $W_c$ as indicated in Figure~\ref{fig-param}. Let us denote by $u^0_j$ and $u_j$ the incident field and respectively the total field associated with a source $J$ at coil position $\x_j \in \Gamma_s$. These fields are solutions to \eqref{axissymmetric-eddycurrent-equation} with $\sigma^{0}$ and $\sigma$ respectively i.e. in the absence and in the presence of a deposit. The field created by a coil at position $z_j$ and recorded by a coil at position $z_i$  is then given by \cite{auld1999review,Blitz1997} 
\begin{equation}
 \mathcal{Z}_{i,j} =\dfrac{i\omega}{r_{0}}\int_{D} (\sigma-\sigma^{0}) u_i(\x) {u^{0}_j(\x)} r \d\x.
\label{zM}
\end{equation}
This impedance has a similar structure as \eqref{us_formula} and corresponds with averaging the scattering field over the recording coil region. Using the reciprocity relation we now design the right hand side of the sampling equation as
$$
\phi_\ell(i) := u^0_i(\zs_\ell).
$$
The inversion algorithm is then performed the same way as in the case of point sources. In the case of a limited number of coils $M$, the inversion is adapted as above by replacing $\mathcal{Z}$ with $\mathcal{Z}^M$ defined by \eqref{zM}.

\subsubsection*{Numerical examples.}

For the numerical examples below, we use realistic physical parameters related to SAX probe system, namely a coil of dimension 0.67mm$\times$2mm (width$\times$height) and a low frequency about $200\pi$ Hertz. The electric conductivity and the magnetic permeability of different compartments involved in our numerical simulation are reported in table~\ref{PhysicalProperty}.

\begin{table}[ht]
\begin{center}
 \begin{tabular}{ccc}
    \hline
    \multicolumn{2}{c}{Electric conductivity in Siemens per meter (S/m)}&\\
    \cline{1-2}
    \hbox{Vacuum}&0.0\\
    \hbox{Tube}&0.97E03\\
    \hbox{Deposit}&1.75E03\\
    \hline
\end{tabular}\vspace{0.1in}\\
\begin{tabular}{ccc}
    \hline
    \multicolumn{2}{c}{Magnetic permeability in Henry per meter (H/m)}&\\
    \cline{1-2}
    \hbox{Vacuum}&4.0E-07$\pi$\\
    \hbox{Tube}&4.04E-07$\pi$\\
    \hbox{Deposit}&4.04E-07$\pi$\\
    \hline
\end{tabular}
\end{center}
\caption{Physical Electrical and Magnetic properties of the Materials used in the numerical Simulations.}
\label{PhysicalProperty}
\end{table}

\subsubsection*{Example of reconstructions obtained with $\mathcal{Z}$ for different values of $N$.}
As observed earlier for point sources, the increase in coil's number enhances qualitatively the inversion that detects both the position and the bulk of the deposit. This is illustrated with the numerical results reported in Figure \ref{dtype_1_COIL}. Indeed, with the increase of point sources from $2^2$, $2^3$ to $2^5$ points, our algorithm can locate the deposit and also gave fairly precise qualitative information about its shape taken here as a semi-disk. One here wants to add as many points sources as possible to enhance the inversion and get a clearer image of the deposition, although, this goes with the expenses of the computational cost mainly related to the increase in the size of the full linear system \eqref{reguSystem}. It is worth noticing, although, that there is, admittedly, a limitation on the reconstruction of the deposit while increasing the number of point sources. The spread of the field generated by faraway sources could barely be detected by other faraway sources. The spread is indeed a function of the distance, the electric conductivity, and the magnetic permeability of the material where the wave travels through. A practical consideration suggests that the zone covered by the set of point sources should cover enough space susceptible to contain the deposit. For this reason, we shall consider, in our numerical experiments $2^5$ point sources and proceed with $\mathcal{Z}^M$, a sparse version of the main matrix as explained above.

\begin{figure}[htbp]
\includegraphics[width=3cm, height=9cm]{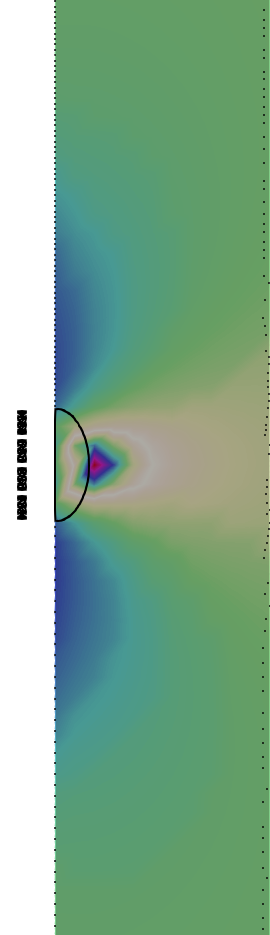}\hfill
\includegraphics[width=3cm, height=9cm]{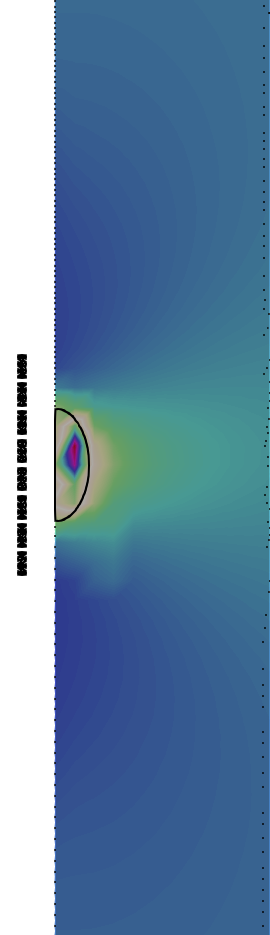}\hfill
\includegraphics[width=3cm, height=9cm]{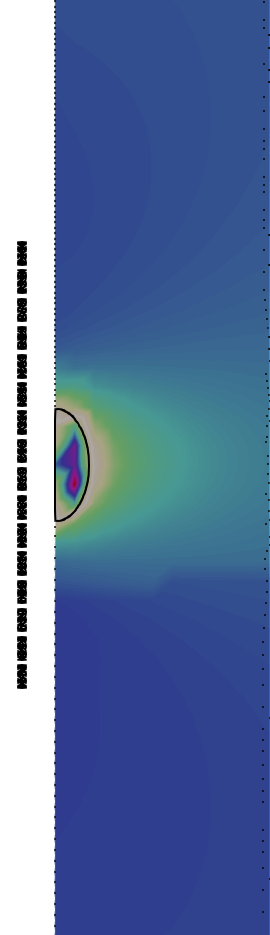}
\caption{Reconstructions obtained for different number of coils $N$. From left to right, $N=4, \, 8, \, 16$. The location of the coils is indicated with black small circles and the exact shape of the deposit  is indicated  using a solid line.}
\label{dtype_1_COIL}
\end{figure}

\subsubsection*{Examples of reconstructions obtained with $\mathcal{Z}^M$ for different values of $M$.}

Figure \ref{dtype_1_PSRC_prob} presents the reconstruction of the deposit using the sparse matrix $\mathcal{Z}^M$ where the total scan is done with $2^5$ point sources. In these results, we vary the number of point sources sliding along the z-direction. It is shown that this procedure, even if it dismisses some measurements (from faraway point sources) the inversion is still capable of retrieving qualitative results comparing to the use of the full matrix $\mathcal{Z}$, which plot is depicted in Figure \ref{dtype_1_PSRC}. These promising results led us to consider the realistic case where point sources become coils-probes. 

\begin{figure}[htbp]
\includegraphics[width=3cm, height=9cm]{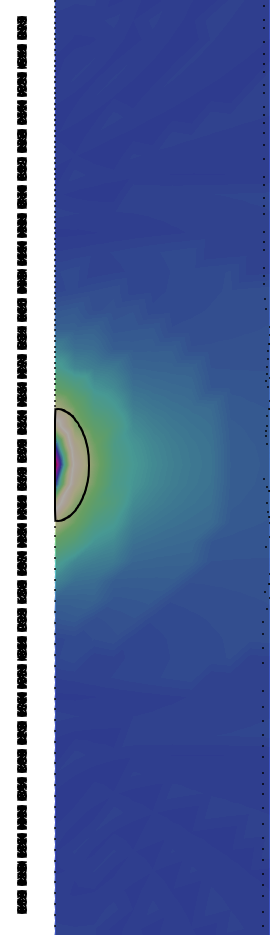}
\hfill
\includegraphics[width=3cm, height=9cm]{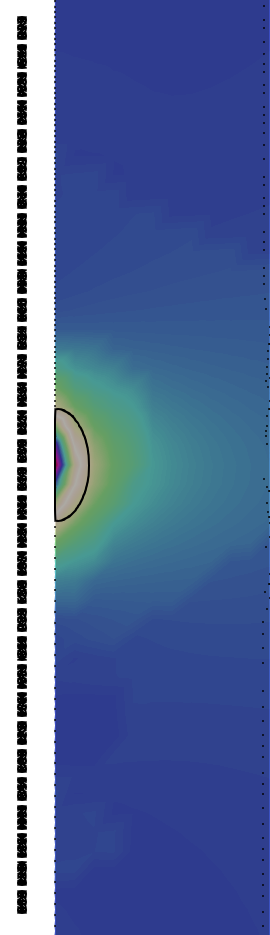}
\hfill
\includegraphics[width=3cm, height=9cm]{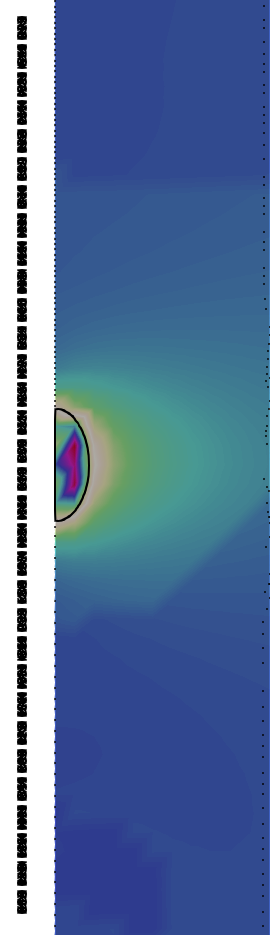}
\caption{Reconstructions obtained with $\mathcal{Z}^M$ for different number of coils  $M$ sliding along the $z-axis$. From left to right, $M=1, \, 2, \, 8$. The location that can be taken by a  coil is indicated with black small circles($N=2^5$ and the exact shape of the deposit is indicated with a solid line and it represents a semi-disc with an elliptical shape of radius 3mm in the $y$-direction and 5mm in the $z$-direction and it represents a semi-disc with an elliptical shape of radius 3mm in the $y$-direction and 5mm in the $z$-direction.}\label{dtype_1_COIL_prob}
\end{figure}

\begin{figure}[htbp]
\includegraphics[width=3cm, height=9cm]{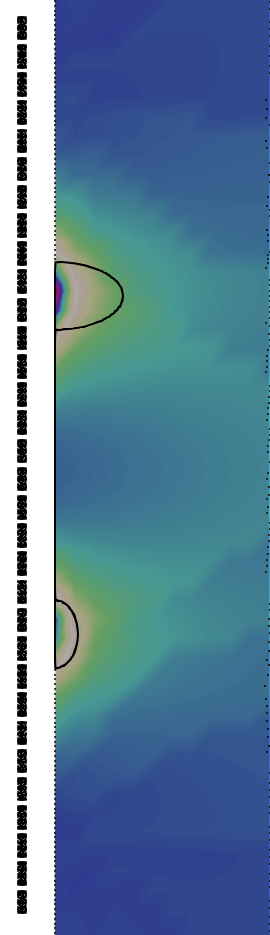}\hfill
\includegraphics[width=3cm, height=9cm]{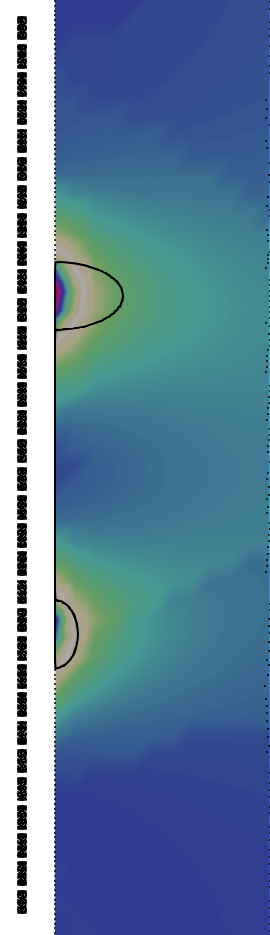}\hfill
\includegraphics[width=3cm, height=9cm]{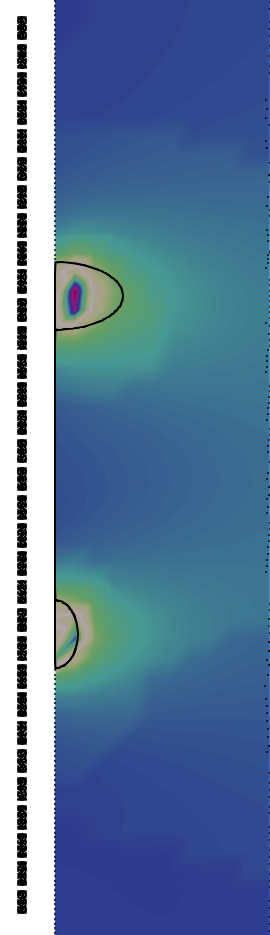}
\caption{Reconstructions obtained with $\mathcal{Z}^M$ for different number of coils  $M$ sliding along the $z-axis$. From left to right, $M=1, \, 2, \, 8$. The location that can be taken by a  coil is indicated with black small circles($N=2^5$ and the exact shape of the deposit is indicated with a solid line and it represents a semi-disc with an elliptical shape of radius 3mm in the $y$-direction and 5mm in the $z$-direction.}\label{dtype_21_COIL_prob}
\end{figure}

Figures \ref{dtype_1_COIL} and \ref{dtype_1_COIL_prob} reproduce similar experiments to those presented in Figures \ref{dtype_1_PSRC} and \ref{dtype_1_PSRC_prob} respectively. The outcome of the reconstruction has the same trends i.e. i) augmentation of the number of probes enhances the quality of the reconstruction and ii) the reconstruction with the limited data keeps producing good images of the deposit. Surprisingly,  the case of back-scattering has produced a clear-cut image of the deposit comparable to the one obtained with the full matrix $\mathcal{Z}$. These promising results have challenged us to consider more complicated situations, such that including several deposits with shapes imitating the drop of water (modeling clogs).

\begin{figure}[htbp]
\includegraphics[width=3cm, height=9cm]{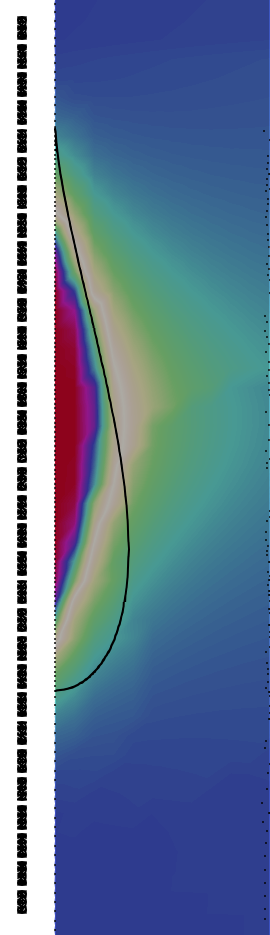}\hfill
\includegraphics[width=3cm, height=9cm]{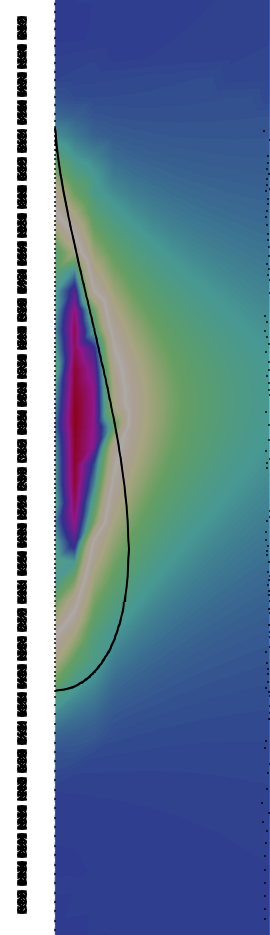}\hfill
\includegraphics[width=3cm, height=9cm]{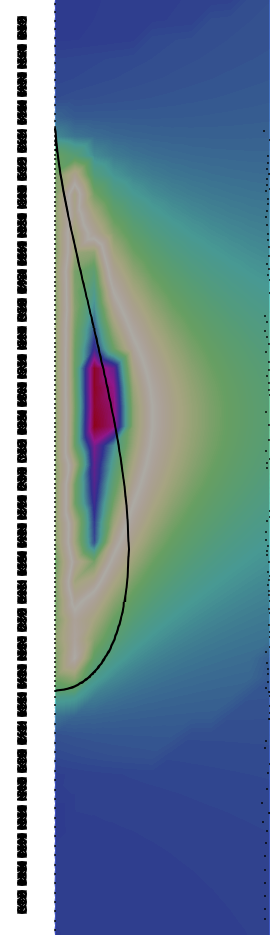}
\caption{Reconstructions obtained with $\mathcal{Z}^M$ for different number of coils  $M$ sliding along the $z-axis$. From left to right, $M=1, \, 2, \, 8$. The location that can be taken by a  coil is indicated with black small rectangle($N=2^5$ and the exact shape of the deposit is indicated with a solid line. The considered shape is about $50$mm long and $8$mm width.}\label{dtype_31_COIL_prob}
\end{figure}

More precisely, Figure \ref{dtype_21_COIL_prob} considers two distanced deposits with different shapes. The inversion algorithm shows that it can detect these two obstacles and performs better as the number of probes sliding along the z-direction gets increased. Figure \ref{dtype_31_COIL_prob} illustrates the same outcome and performance of the inversion algorithm. In this case, we consider a deposit with a shape that models a drop (relatively long, i.e., approximately as long as 12 coils). It is shown that the inversion algorithm can detect the bulk of the deposit and also approaches its shape.

\section*{Conclusion}
We presented in this work an eddy-current inverse shape problem and employed the so-called Linear Sampling Method to propose an imaging algorithm. We have tested and thoroughly analyzed the method both in theoretical and industrial settings, respectively. We numerically showed in particular that even for (nearly) back-scattering data, a commonly encountered configuration in practice, an adaptation of the algorithm leads to satisfying results. The latter can serve to provide fast qualitative inspection of a large number of tubes. The outcome can also be used as initial guess for more computationally involved  inversion methods based on optimization techniques. Exploring this coupling for 3D configurations is one of the perspectives of this work. We are also interested in studying the performance this procedure for identifying defects with different types, such as cracks hidden by deposits, which constitutes one of the main current challenges for the considered industrial application.

%~~~~~~~~~~~~~~~~~~~~~~~~~~~~~~~~~~~~~~~~~~~~~~~~~~
\section*{References}
%~~~~~~~~~~~~~~~~~~~~~~~~~~~~~~~~~~~~~~~~~~~~~~~~~~
%\bibliographystyle{iopart-num}
\bibliographystyle{plain}
\bibliography{biblio.bib}
\end{document}